\documentclass{amsart}

\usepackage{verbatim,url,epsfig,amssymb}
\usepackage[active]{srcltx}
\usepackage[all]{xy}
\usepackage[a4paper,hcentering,vcentering]{geometry}
\newtheorem{theorem}{Theorem}

\newtheorem{lemma}{Lemma}

\newcommand{\C}{\ensuremath{\mathbb{C}}}

\newcommand{\R}{\ensuremath{\mathbb{R}}}
\newcommand{\T}{\ensuremath{\mathbb{T}}}
\newcommand{\Z}{\ensuremath{\mathbb{Z}}}

\newcommand{\Tr}{\ensuremath{\mathrm{T}}}

\renewcommand{\Im}{\ensuremath{\mathrm{Im}\,}}

\newcommand{\dk}{\ensuremath{\mathfrak{d}}}
\newcommand{\gk}{\ensuremath{\mathfrak{g}}}

\DeclareMathOperator{\id}{id}

\begin{document}

\title{A Proof of the Invariant Torus Theorem of Kolmogorov}

\author{Jacques F{\'e}joz}


\begin{abstract}
  The invariant torus theorem is proved using a simple fixed point
  theorem.
\end{abstract}

\maketitle

Let $\mathcal{H}$ be the space of germs along $\Tr_0^n := \T^n \times
\{0\}$ of real analytic Hamiltonians in $\T^n \times \R^n =
\{(\theta,r)\}$ ($\T^n=\R^n/\Z^n$), endowed with the usual, inductive
limit topology (see section~\ref{sec:action}). The vector field
associated with $H \in \mathcal{H}$ is
$$\vec H : \quad \dot\theta = \partial_r H, \quad \dot
r=-\partial_\theta H.$$ 

For $\alpha\in\R^n$, let $\mathcal{K}^\alpha$ be the affine subspace
of Hamiltonians $K\in \mathcal{H}$ such that $K|_{\Tr_0^n}$ is
constant (i.e.  $\Tr_0^n$ is invariant) and $\vec
K|_{\Tr_0^n}=\alpha$:
$$\mathcal{K}^\alpha = \{ K \in \mathcal{H}, \; \exists c \in \R, \;
K (\theta,r) = c + \alpha \cdot r + O(r^2)\}, \quad \alpha \cdot r =
\alpha_1 r_1 + \cdots + \alpha_n r_n,$$ where $O(r^2)$ are terms of
the second ordrer in $r$, which depend on $\theta$. 

Let also $\mathcal{G}$ be the space of germs along $\Tr_0^n$ of real
analytic symplectomorphisms $G$ in $\T^n \times \R^n$ of the following
form:
$$G(\theta,r) = (\varphi(\theta), (r+\rho(\theta)) \cdot
\varphi'(\theta)^{-1}),$$ where $\varphi$ is an isomorphism of $\T^n$
fixing the origin (meant to straighten the flow on an invariant
torus), and $\rho$ is a closed $1$-form on $\T^n$ (meant to straighten
an invariant torus).

In the whole paper we fix $\alpha \in \R^n$ Diophantine ($0< \gamma
\ll 1 \ll \tau$; see \cite{Poschel:2001}):
$$|k \cdot \alpha| \geq \gamma |k|^{-\tau} \quad (\forall k \in \Z^n
\setminus \{0\}), \quad |k|= |k_1| + \cdots + |k_n|$$ 
and 
$$K^o(\theta,r) = c^o + \alpha \cdot r + Q^o(\theta) \cdot r^2 +
O(r^3) \in \mathcal{K}^\alpha$$ such that the average of the quadratic
form valued function $Q^o$ be non-degenerate:
$$\det \int_{\T^n} Q^o(\theta) \, d\theta \neq 0.$$ 


\begin{theorem}[Kolmogorov \cite{Kolmogorov:1954, Chierchia:2008bis}]
  \label{thm:kolmogorov} For every $H\in\mathcal{H}$ close to $K^o$,
  there exists a unique $(K,G) \in \mathcal{K}^\alpha \times
  \mathcal{G}$ close to $(K^o,\id)$ such that $H = K \circ G$ in some
  neighborhood of $G^{-1}(\Tr_0^n)$.
\end{theorem}

See \cite{Poschel:2001, Sevryuk:2003} and references therein for
background. The functional setting below is related
to~\cite{Fejoz:2010}.

\section{The action of a group of symplectomorphisms}
\label{sec:action}

Define complex extensions $\T^n_\C = \C^n / \Z^n$ and $\Tr^n_\C =
\T^n_\C \times \C^n$, and neighborhoods ($0<s<1$)
$$\T^n_s = \{\theta \in \T^n_\C, \; \max_{1\leq j\leq n} |\Im
\theta_j| \leq s\} \quad \mbox{and} \quad \Tr^n_s = \{ (\theta,r)\in
\Tr^n_\C, \, \max_{1\leq j \leq n} \max \left( |\Im
  \theta_j|, |r_j| \right) \leq s\}.$$ 

For complex extensions $U$ and $V$ of real manifolds, denote by
$\mathcal{A}(U,V)$ the Banach space of real holomorphic maps from the
interior of $U$ to $V$, which extend continuously on $U$;
$\mathcal{A}(U):=\mathcal{A}(U,\C)$.

\bigskip $\bullet$ Let $\mathcal{H}_s = \mathcal{A}(\Tr^n_s)$ with norm
$|H|_s := \sup_{(\theta,r)\in \Tr^n_s} |H(\theta,r)|$, such that
$\mathcal{H} = \cup_s \mathcal{H}_s$ be their inductive limit.

Fix $s_0$. There exist $\epsilon_0$ such that $K^o \in
\mathcal{H}_{s_0}$ and, for all $H \in \mathcal{H}_{s_0}$ such that
$|H-K^o|_{s_0} \leq \epsilon_0$,
\begin{equation}
  \label{eq:det}
  \left|\det \int_{\T^n} \frac{\partial^2 H}{\partial r^2}(\theta,0) \,
    d\theta \right| \geq \frac{1}{2} \left|\det \int_{\T^n}
    \frac{\partial^2 K^o}{\partial r^2}(\theta,0) \, d\theta \right|
  \neq 0.
\end{equation}

Hereafter we assume that $s$ is always $\geq s_0$. Set
$\mathcal{K}_s^\alpha = \{K \in \mathcal{H}_s \cap \mathcal{K}^\alpha,
\; |K-K^o|_{s_0} \leq \epsilon_0\}$, and let $\mathcal{\vec K}_s
\equiv \R \oplus O(r^2)$ be the vector space directing
$\mathcal{K}^\alpha_s$.

$\bullet$ Let $\mathcal{D}_s$ be the space of isomorphisms $\varphi\in
\mathcal{A}(\T^n_s,\T^n_\C)$ with $\varphi(0)=0$ and $\mathcal{Z}_s$
be the space of bounded real holomorphic closed $1$-forms on $\T^n_s$.
The semi-direct product $\mathcal{G}_s = \mathcal{Z}_s \rtimes
\mathcal{D}_s$ acts faithfully and symplectically on the phase space
by
\begin{equation}
  \label{eq:action}
  G : \T^n_s \rightarrow \T^n_\C, \quad (\theta,r) \mapsto
  (\varphi(\theta), (\rho(\theta)+r) \cdot \varphi'(\theta)^{-1}), \quad
  G=(\rho,\varphi),
\end{equation}
and, to the right, on $\mathcal{H}_s$ by $\mathcal{H}_s \rightarrow
\mathcal{A}(G^{-1}(\T^n_s))$, $K \mapsto K \circ G$.

$\bullet$ Let $\dk_s :=\{\dot\varphi \in \mathcal{A}(\T^n_s)^n,
\;\dot\varphi(0)=0\}$ with norm $|\dot\varphi|_s :=
\max_{\theta\in\T^n_s} \max_{1\leq j \leq n} |\dot\varphi_j(\theta)|$,
be the space of vector fields on $\T^n_s$ which vanish at $0$.
Similarly, let $|\dot \rho|_s = \max_{\theta\in\T^n_s} \max_{1\leq j
  \leq n}, |\dot \rho_j(\theta)|$ on $\mathcal{Z}_s$.  An element
$\dot G = (\dot\rho,\dot\varphi)$ of the Lie algebra $\mathfrak{g}_s =
\mathcal{Z}_s \oplus \dk_s$ (with norm $|(\dot\rho,\dot\varphi)|_s =
\max(|\dot\rho|_s,|\dot\varphi|_s)$) identifies with the vector field
\begin{equation}
  \label{eq:dot}
  \dot G : \T^n_s \rightarrow \C^n, \quad (\theta,r) \mapsto
  (\dot\varphi(\theta), \dot\rho(\theta) - r \cdot
  \dot\varphi'(\theta)), 
\end{equation}
whose exponential is denoted by $\exp\dot G$. It acts infinitesimally
on $\mathcal{H}_s$ by $\mathcal{H}_s \rightarrow \mathcal{H}_s$, $K
\mapsto K'\cdot \dot G$.

Constants $\gamma_i,\tau_i,c_i,t_i$ below do not depend on $s$ or
$\sigma$.

\setcounter{lemma}{-1}
\begin{lemma} \label{lm:exp} If $\dot G \in \gk_{s+\sigma}$ and $|\dot
  G|_{s+\sigma} \leq \gamma_0 \sigma^2$, then $\exp \dot G \in
  \mathcal{G}_s$ and $|\exp \dot G -\id|_s \leq c_0 \sigma^{-1} |\dot
  G|_{s+\sigma}$.
\end{lemma}

\begin{proof}
  Let $\chi_s = \mathcal{A}(\T^n_s)^{2n}$, with norm $\|v\|_s
  =\max_{\theta\in \T^n_s} \max_{1\leq j \leq n} |v_j(\theta)|$.  Let
  $\dot G \in \gk_{s+\sigma}$ with $|\dot G|_{s+\sigma} \leq \gamma_0
  \sigma^2$, $\gamma_0 := (36n)^{-1}$.  Using definition~(\ref{eq:dot}) and
  Cauchy's inequality, we see that if $\delta:=\sigma/3$,
  $$\|\dot G\|_{s+2\delta} = \max \left(| \dot\varphi
    |_{s+2\delta}, |\dot\rho + r \cdot
    \dot\varphi'(\theta)|_{s+2\delta} \right) \leq 2n\delta^{-1} |\dot
  G|_{s+3\delta} \leq \delta/2.$$

  Let $D_s = \{t\in \C, \; |t|\leq s\}$ and $F := \left\{f \in
    \mathcal{A}(D_s \times \T^n_s)^{2n}, \; \forall (t,\theta) \in D_s
    \times \T^n_s,\; |f(t,\theta)|_s \leq \delta \right\}$. By
  Cauchy's inequality, the Lipschitz constant of the Picard operator
  $$P : F\rightarrow F, \quad f \mapsto Pf, \quad (Pf)(t,\theta) =
  \int_0^t \dot G(\theta + f(s,\theta)) \, ds$$ is $\leq 1/2$.  Hence,
  $P$ possesses a unique fixed point $f \in F$, such that $f(1,\cdot)
  = \exp(\dot G) - \id$ and $|f(1,\cdot)|_s\leq \|\dot G\|_{s+\delta}
  \leq c_0 \sigma^{-1}|\dot G|_{s+\sigma}$, $c_0 = 6n$.
  
  Also, $\exp \dot G \in \mathcal{G}_s$ because at all times the curve
  $\exp(t \dot G)$ is tangent to $\mathcal{G}_s$, locally a closed
  submanifold of $\mathcal{A}(\T^n_s,\T^n_\C)$ (the method of the
  variation of constants gives an alternative proof).
\end{proof}

\section{A property of infinitesimal transversality } 
\label{sec:transversality}

We will show that locally $\mathcal{\vec K}_s$ is tranverse to the
infinitesimal action of $\mathfrak{g}_s$ on $\mathcal{H}_{s+\sigma}$.

\begin{lemma} \label{lm:1} For all $(K, \dot H) \in
  \mathcal{K}^\alpha_{s+\sigma} \times \mathcal{H}_{s+\sigma}$, there
  exists a unique $(\dot K,\dot G) \in \mathcal{\vec K}_s \times
  \mathfrak{g}_s$ such that
  $$\dot K + K' \cdot \dot G = \dot H \quad \mbox{and} \quad \max
  (|\dot K|_s,|\dot G|_s) \leq c_1 \sigma^{-t_1} \left(1 +
    |K|_{s+\sigma} \right)|\dot H|_{s+\sigma}.$$
\end{lemma}

\begin{proof}
  We want to solve the linear equation $\dot K + K' \cdot \dot G =
  \dot H$. Write 
  $$
  \begin{cases}
    K(\theta,r) = c + \alpha \cdot r + Q(\theta) \cdot r^2 + O(r^3)\\
    \dot K (\theta,r) = \dot c + \dot K_2(\theta,r), &\dot c\in \R,
    \quad \dot K_2 \in O(r^2)\\
    \dot G(\theta,r) = (\dot \varphi(\theta), R + S'(\theta) - r \cdot
    \dot\varphi'(\theta)), \quad &\dot\varphi \in \chi_s, \quad \dot R
    \in \R^n, \quad \dot S \in \mathcal{A}(\T^n_s).
  \end{cases}
  $$
  Expanding the equation in powers of $r$ yields
  \begin{equation}
    \label{eq:lin}
    \left(\dot c + (\dot R + \dot S') \cdot \alpha\right) + r \cdot
    \left(-\dot\varphi' \cdot \alpha + 2 Q \cdot (\dot R + \dot
      S')\right) + \dot K_2 = \dot H =: \dot H_0 + \dot H_1 \cdot r +
    O(r^2),
  \end{equation}
  where the term $O(r^2)$ on the right hand side does not depend on
  $\dot K_2$.


  Fourier series and Cauchy's inequality show that if $g \in
  \mathcal{A}(\T^n_{s+\sigma})$ has zero average, there is a unique
  function $f \in \mathcal{A}(\T^n_s)$ of zero average such that
  $L_\alpha f := f' \cdot \alpha = g$, and $|f|_s \leq c
  \sigma^{-t} |g|_{s+\sigma}$~\cite{Poschel:2001}.

  Equation~(\ref{eq:lin}) is triangular in the unknowns and
  successiveley yields:
  $$
  \begin{cases}
    \dot S &= L_\alpha^{-1} \left(\dot H_0 - \int_{\T^n} \dot
      H_0(\theta)\, d\theta \right)\\
    \dot R &= \frac{1}{2} \left(\int_{\T^n} Q(\theta)\,
      d\theta\right)^{-1} \int_{\T^n} \left( \dot H_1(\theta) - 2
      Q(\theta) \cdot \dot S'(\theta) \right) \, d\theta\\
    \dot \varphi &= L_\alpha^{-1} \left( \dot H_1(\theta) - 2
      Q(\theta)
      \cdot (\dot R + \dot S'(\theta)) \right)\\
    \dot c &= \int_{\T^n} \dot H_0 (\theta) \, d\theta - \dot R \cdot
    \alpha\\
    \dot K_2 &= O(r^2),
  \end{cases}
  $$
  and, together with Cauchy's inequality, the wanted estimate.
\end{proof}


\section{The local transversality property}

Let us bound the discrepancy between the action of $\exp(- \dot G)$ and
the infinitesimal action of $-\dot G$.

\begin{lemma} \label{lm:2} For all $(K, \dot H) \in
  \mathcal{K}^\alpha_{s+\sigma} \times \mathcal{H}_{s+\sigma}$ such
  that $(1+|K|_{s+\sigma})|\dot H|_{s+\sigma} \leq \gamma_2
  \sigma^{\tau_2}$, if $(\dot K, \dot G) \in \mathcal{\vec K} \times
  \gk_s$ solves the equation $\dot K + K' \circ \dot G = \dot H$
  (lemma~\ref{lm:1}), then $\exp \dot G \in \mathcal{G}_s$, $|\exp\dot
  G - \id|_s \leq \sigma$ and
  $$|(K+\dot H) \circ \exp(-\dot G)-(K+\dot K)|_s\leq c_2 \sigma^{-t_2}
  (1+|K|_{s+\sigma})^2|\dot H|_{s+\sigma}^2.$$
\end{lemma}

\begin{proof}
  Set $\delta = \sigma/2$. Lemmas~\ref{lm:exp} and~\ref{lm:1} show
  that, under the hypotheses for some constant $\gamma_2$ and for
  $\tau_2=t_1+1$, we have $|\dot G|_{s+\delta} \leq \gamma_0 \delta^2$
  and $|\exp\dot G - \id|_{s} \leq \delta$.

  Let $H = K + \dot H$. Taylor's formula says
  $$\mathcal{H}_s \ni H \circ \exp (- \dot G) = H - H' \cdot \dot G +
  \left(\int_0^1 (1-t) \, H'' \circ \exp (- t \dot G) \, dt \right)
  \cdot \dot G^2$$ or, using the fact that $H = K + \dot K + K' \cdot
  \dot G$,
  $$H \circ \exp (- \dot G) - (K+\dot K) = - (\dot K + K' \cdot \dot
  G)' \cdot \dot G + \left( \int_0^1 (1-t) \, H'' \circ \exp (- t \dot
    G) \, dt \right) \cdot \dot G^2.$$ The wanted estimate thus
  follows from the estimate of lemma~\ref{lm:1} and Cauchy's
  inequality.
\end{proof}

Let $B_{s,\sigma} = \{(K,\dot H) \in \mathcal{K}^\alpha_{s+\alpha}
\times \mathcal{H}_{s+\sigma}, \; |K|_{s+\sigma} \leq \epsilon_0, \;
|\dot H|_{s+\sigma} \leq (1+\epsilon_0)^{-1}\gamma_2
\sigma^{\tau_2}\}$ (recall~(\ref{eq:det})).

\begin{minipage}[c]{0.6\linewidth}
  According to lemmas~\ref{lm:1}-\ref{lm:2}, the map $\phi :
  B_{s,\sigma} \rightarrow \mathcal{K}^\alpha_s \times \mathcal{H}_s$,
  $$\phi(K,\dot H) = (K + \dot K, (K+\dot H) \circ \exp (-\dot G) -
  (K+\dot K)),$$ 
  satisfies, if $(\hat K, \widehat{\dot H}) = \phi(K,\dot
  H)$, 
  $$|\hat K - K|_s, |\widehat{\dot H}|_s \leq c_3\sigma^{-t_3} |\dot
  H|_{s+\sigma}^2.$$ Theorem~\ref{thm:fp} applies and shows that if
  $H-K^o$ is small enough in $\mathcal{H}_{s+\sigma}$, the sequence
  $(K_j, \dot H_j) = \phi^j(K^o,H-K^o)$, $j\geq 0$, converges towards
  some $(K,0)$ in $\mathcal{K}^\alpha_s \times \mathcal{H}_s$.
\end{minipage}
\hspace{2mm}
\begin{minipage}[c]{0.3\linewidth}
  \includegraphics[scale=0.6]{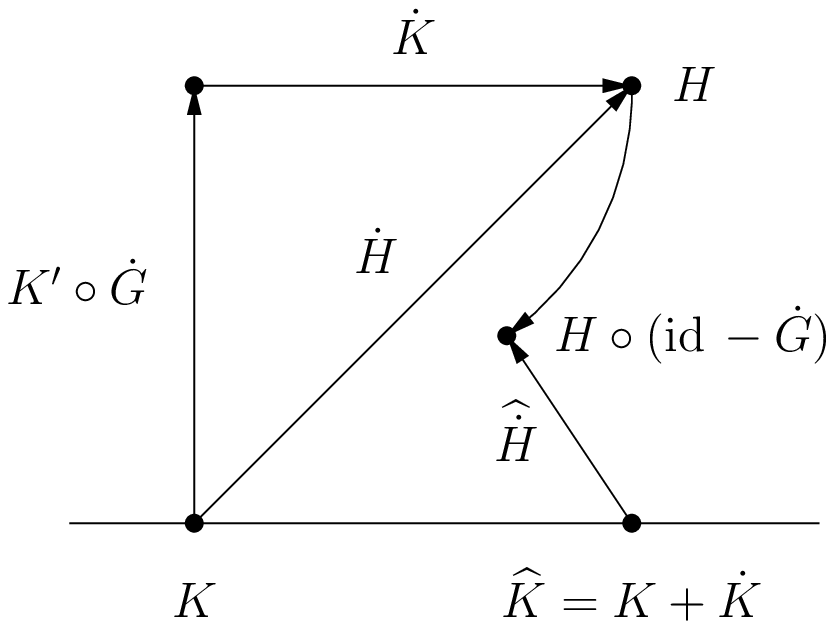}
\end{minipage}

Let us keep track of the $\dot G_j$'s solving with the $\dot K_j$'s
the successive linear equations $\dot K_j + K_j' \cdot \dot G_j = \dot
H_j$ (lemma~\ref{lm:1}). At the limit,
$$K := K^o + \dot K_0 + \dot K_1 + \cdots = H \circ \exp(-\dot G_0)
\circ \exp (-\dot G_1) \circ \cdots.$$ Moreover, lemma~\ref{lm:1}
shows that $|\dot G_j|_{s_{j+1}} \leq c_4 \sigma_j^{-t_4}|\dot
H_j|_{s_j}$, hence the isomorphisms $\gamma_j := \exp(-\dot G_0) \circ
\cdots \circ \exp (- \dot G_j)$, which satisfy
$$|\gamma_n-\id|_{s_{n+1}} \leq |\dot G_0|_{s_1} + ... + |\dot
G_n|_{s_{n+1}},$$ form a Cauchy sequence and have a limit $\gamma \in
\mathcal{G}_s$. At the expense of decreasing $|H-K^o|_{s+\sigma}$, by
the inverse function theorem, $G:= \gamma^{-1}$ exists in
$\mathcal{G}_{s-\delta}$ for some $0 < \delta <s$, so that $H = K
\circ G$.

\section*{Appendix. A fixed point theorem}
\label{app:fixed-point-theorem}

Let $(E_s,| \cdot |_s)_{0<s<1}$ and $(F_s,| \cdot |_s)_{0<s<1}$ be two
decreasing families of Banach spaces with increasing norms. On $E_s
\times F_s$, set $|(x,y)|_s = \max (|x|_s, |y|_s)$. Fix
$C,\gamma,\tau,c,t>0$. 

Let 
$$\phi : B_{s,\sigma} := \left\{ (x,y) \in E_{s+\sigma} \times
  F_{s+\sigma}, \; |x|_{s+\sigma} \leq C, \; |y|_{s+\sigma} \leq
  \gamma \sigma^\tau \right\} \rightarrow E_s \times F_s$$ be a family
of operators commuting with inclusions, such that if $(X,Y) =
\phi(x,y)$,
$$|X-x|_s, \; |Y|_s \leq c \sigma^{-t} |y|_{s+\sigma}^2.$$

In the proof of theorem~\ref{thm:kolmogorov}, ``$|x|_{s+\sigma}\leq
C$'' allows us to bound the determinant of $\int_{\T^n} Q(\theta)
d\theta$ away from $0$, while ``$|y|_{s+\sigma} \leq \gamma
\sigma^\tau$'' ensures that $\exp \dot G$ is well defined.

\begin{theorem} \label{thm:fp}
  Given $s<s+\sigma$ and $(x,y) \in B_{s,\sigma}$ such that
  $|(x,y)|_{s+\sigma}$ is small, the sequence $(\phi^j(x,y))_{j\geq 0}$
  exists and converges towards a fixed point $(\xi,0)$ in $B_{s,0}$.
\end{theorem}

\begin{proof}
  It is convenient to first assume that the sequence is defined and
  $(x_j,y_j) := F^j(x,y) \in B_{s_j,\sigma_j}$, for $s_j := s +
  2^{-j}\sigma$ and $\sigma_j := s_{j} - s_{j+1}$. We may assume
  $c\geq 2^{-t}$, so that $d_j := c \sigma_j^{-t} \geq 1$. By
  induction, and using the fact that $\sum 2^{-k} = \sum k2^{-k} = 2$,
  $$|y_j|_{s_j} \leq d_{j-1} |y_{j-1}|_{s_{j-1}}^2
  \leq \cdots \leq |y|_{s+\sigma}^{2^j} \prod_{0 \leq k \leq j-1}
  d_k^{2^{k+1}} \leq \Bigl( |y|_{s+\sigma} \prod_{k \geq 0}
  d_k^{2^{-k-1}} \Bigr)^{2^j} = \bigl( c 4^t \sigma^{-t}
  |y|_{s+\sigma} \bigr)^{2^j}.$$ Given that $\sum_{n\geq 0} \mu^{2^n}
  \leq 2\mu$ if $2\mu\leq 1$, we now see by induction that if
  $|(x,y)|_{s+\sigma}$ is small enough, $(x_j,y_j)$ exists in
  $B_{s_j,\sigma_j}$ for all $j\geq 0$, $y_j$ converges to $0$ in
  $F_s$ and the series $x_j = x_0 + \sum_{0\leq k\leq j-1}
  (x_{k+1}-x_k)$ converges normally towards some $\xi \in E_s$ with
  $|\xi|_s \leq C$. 
\end{proof}

\bigskip \textit{Thank you to Ivar Ekeland, {\'E}ric S{\'e}r{\'e} and Alain
  Chenciner for pointing out a mistake and for illuminating
  discussions.}

\bibliographystyle{plain}
\bibliography{index} 

\end{document}